\documentclass[a4paper,12pt,reqno]{amsart}

\usepackage[headings]{fullpage}

\usepackage{amsfonts,graphics,amsmath,amsthm,amscd,amssymb,latexsym,euscript,enumerate,booktabs,mathrsfs}

\usepackage{epsfig}
\usepackage{flafter}
\usepackage[all,cmtip,line]{xy}
\usepackage{array}
\usepackage[english]{babel}
\usepackage{overpic}
\usepackage{subfig}
\usepackage{multirow}
\usepackage{microtype}
\usepackage{wrapfig}
\usepackage{longtable}
\usepackage{supertabular}
\usepackage{caption}
\usepackage{tikz}
\usetikzlibrary{positioning}
\usepackage{tikz-cd}
\usepackage{float}
\allowdisplaybreaks

\usepackage[dvipsnames,svgnames,table]{xcolor}
\usepackage{graphicx}
\usepackage{hyperref}
\hypersetup{
    colorlinks=true,    
    linkcolor=blue,          
    citecolor=blue,      
    filecolor=blue,      
    urlcolor=blue
}

\newtheorem{theorem}{Theorem}[section]
\newtheorem{lemma}[theorem]{Lemma}
\newtheorem{proposition}[theorem]{Proposition}

\newtheorem{question}[theorem]{Question}
\newtheorem*{theorem*}{Theorem}
\newtheorem*{question*}{Question}

\theoremstyle{plain}
\newtheorem{corollary}[theorem]{Corollary}

\theoremstyle{definition} 

\newtheorem{definition-lemma}[theorem]{Definition-Lemma}

\newtheorem{example}[theorem]{Example}

\newtheorem{remark}[theorem]{Remark}
\numberwithin{equation}{section}

\newcommand{\C}{\mathbb{C}}
\newcommand{\R}{\mathbb{R}}
\newcommand{\K}{\mathbb{K}}

\newcommand{\Q}{\mathbb{Q}}
\newcommand{\OO}{\mathcal{O}}

\newcommand{\bM}{\mathbf M}

\def\P{\mathbb{P}}

\newcommand{\A}{\mathbb{A}}
\newcommand{\F}{\mathbb{F}}
\DeclareMathOperator{\ord}{ord}

\DeclareMathOperator{\Hom}{Hom}
\DeclareMathOperator{\Cent}{Cent}

\DeclareMathOperator{\Stab}{Stab}

\DeclareMathOperator{\Pic}{Pic}

\DeclareMathOperator{\Nm}{Nm}

\DeclareMathOperator{\im}{im}

\def\Pic{\operatorname{Pic}}

\def\Spec{\operatorname{Spec}}

\def\mult{\operatorname{mult}}

\def\cha{\operatorname{char}}

\DeclareMathOperator{\divisor}{div}
\DeclareMathOperator{\trdeg}{trdeg}

\def\Cent{\operatorname{Cent}}
\def\Supp{\operatorname{Supp}}

\DeclareMathOperator{\Gal}{Gal}

\def\Nklt{\operatorname{Nklt}}

\usepackage{mathtools}

\DeclarePairedDelimiterX{\norm}[1]{\lVert}{\rVert}{#1}

\title[Semiample perturbations over arbitrary fields]
{Semiample perturbations over arbitrary fields}
\author[D.-W. Lee]{Dae-Won Lee}
\address{School of Mathematics, Korea Institute for Advanced Study,
85 Hoegiro, Dongdaemun-gu, Seoul 02455, Republic of Korea}
\email{daewonlee@kias.re.kr}
\thanks{The author is supported by a KIAS Individual Grant (MG111201)
at Korea Institute for Advanced Study.}
\date{\today}
\subjclass[2020]{14G17, 14J30, 14E30}
\keywords{Semiample divisor, log canonical pair, imperfect field, generalized pair}

\begin{document}

\begin{abstract}
Assuming projective log resolutions, we construct semiample perturbations over arbitrary fields that preserve all log canonical places. As an application, for divisor classes with finitely generated section rings, we characterize the existence of effective divisors yielding lc or klt pairs in terms of asymptotic orders of vanishing. We also derive minimal model and semiampleness results for generalized threefolds over $F$-finite fields of characteristic $p>5$.
\end{abstract}
\maketitle

\section{Introduction}

Let $(X,\Delta)$ be a projective log canonical pair and let $D$ be a semiample $\K$-Cartier $\K$-divisor, where $\K\in\{\Q,\R\}$. In characteristic zero, Bertini's theorem gives an effective divisor $D'\sim_{\K}D$ with $(X,\Delta+D')$ lc. If $(X,\Delta)$ is klt, then $D'$ can be chosen with $(X,\Delta+D')$ klt.

In positive characteristic, a general member of a base point free linear system on a regular variety need not be regular. Over a finite field, a nonempty open subset of a linear system need not contain a member defined over the ground field. Tanaka proved semiample perturbation over an $F$-finite field containing an infinite perfect subfield, assuming a log resolution \cite[Theorem 1]{Tan17}. Tanaka stated that it was not known whether his perturbation theorem holds over finite fields \cite[\S1, ``Assumption on the base field'']{Tan17}. This leads to the following question.

\begin{question}\label{ques:tanaka}
Let $(X,\Delta)$ be a projective lc pair over a finite field admitting a projective log resolution. Let $D$ be a semiample $\K$-Cartier $\K$-divisor, where $\K\in\{\Q,\R\}$ and $\Delta$ is a $\K$-divisor. Does there exist an effective divisor $D'\sim_{\K}D$ such that $(X,\Delta+D')$ is lc? If $(X,\Delta)$ is klt, then can $D'$ be chosen such that $(X,\Delta+D')$ is klt?
\end{question}

Theorem \ref{thm:main} answers Question \ref{ques:tanaka} in a more general setting: it holds over an arbitrary field for projective sub-lc subpairs and gives a uniform discrepancy bound for all prime divisors over $X$.

\begin{theorem}\label{thm:main}
Let $(X,\Delta)$ be a projective subpair over a field $k$. Assume that $(X,\Delta)$ is sub-lc and admits a projective log resolution. Let $D$ be a semiample $\K$-Cartier $\K$-divisor, where $\K\in\{\Q,\R\}$. For every $0<\epsilon<1$, there exists an effective $\K$-Cartier $\K$-divisor $D_\epsilon\sim_{\K}D$ such that
\begin{equation}\label{eq:main-bound}
(1-\epsilon)A_{X,\Delta}(E)
\leq A_{X,\Delta+D_\epsilon}(E)
\leq A_{X,\Delta}(E)
\end{equation}
for every prime divisor $E$ over $X$.
\end{theorem}

The inequalities in \eqref{eq:main-bound} preserve sub-lc and sub-klt singularities and all lc places. Both the divisor $D_\epsilon$ and the linear equivalence $D_\epsilon\sim_{\K}D$ are defined over $k$.

The proof begins with a projective log resolution $\pi\colon Y\to X$ and a reduced simple normal crossing boundary on $Y$. In characteristic $p>0$, we first pass to a separable algebraic extension where Tanaka's theorem applies. After descending the perturbing divisor to a finite Galois extension, we average its Galois conjugates and descend the resulting linear equivalence to the ground field by the norm. Murayama's $\Gamma$-construction \cite[Corollary 3.7(i)]{Mur21} and purely inseparable descent then give the result over an arbitrary field. Perturbing an integer multiple of $\pi^*D$ and dividing by that multiple gives \eqref{eq:main-bound} on the log resolution.

Gongyo--Nakamura--Tanaka proved klt perturbation for nef and big $\Q$-divisors in dimension at most three over perfect fields \cite[Lemma 2.8]{GNT19}. Das proved simultaneous lc and klt perturbation over fields of positive characteristic containing an infinite perfect subfield \cite[Theorem 3.1]{Das20}. Tanaka obtained lc and klt perturbations after a purely transcendental base change \cite[Corollary 4.20]{Tan24}. 

Theorem \ref{thm:main} also gives effective representatives by perturbing the semiample part of a birational decomposition. For a finitely generated section ring, the fixed part on a suitable model computes the \emph{asymptotic order of vanishing}, which is defined as  $v_E(\|D\|)\coloneqq\inf\{\ord_E G\mid G\geq0,\ G\sim_{\Q}D\}$ \cite[Definition 5.5]{Mur21}. Here, $D$ is $\Q$-Cartier with $\kappa(D)\geq0$, and $E$ is a prime divisor over $X$. For $q>0$ with $qD$ Cartier, write $\mathfrak b(|qD|)$ for the base ideal of $|qD|$.

\begin{theorem}\label{thm:fixed}
Let $(X,\Delta)$ be a projective subpair over a field $k$, and let $D$ be a $\Q$-Cartier divisor with $\kappa(D)\geq0$. Assume that there is an integer $q>0$ such that $qD$ is Cartier, $R(X,qD)$ is generated in degree one, and $(X,\Delta,\mathfrak b(|qD|))$ admits a projective log resolution. Then an effective divisor $G\sim_{\Q}D$ with $(X,\Delta+G)$ sub-lc exists if and only if
\begin{equation}\label{eq:criterion}
A_{X,\Delta}(E)\geq v_E(\|D\|)
\qquad\text{for every prime divisor }E\text{ over }X.
\end{equation}
An effective divisor $G\sim_{\Q}D$ with $(X,\Delta+G)$ sub-klt exists if and only if all inequalities in \eqref{eq:criterion} are strict.
\end{theorem}

The inequalities in \eqref{eq:criterion} imply that $(X,\Delta)$ is sub-lc. If strict inequalities hold, then $(X,\Delta)$ is sub-klt. Example \ref{ex:finite-generation} shows that the lc criterion can fail without finite generation. Section \ref{subsubsec:anticanonical} gives anticanonical applications and compares the construction with potential log discrepancies \cite[Proposition 1.4 and Theorem 5.10]{CP16}.

The perturbation theorem also gives ordinary boundaries for generalized pairs. Write $M_X$ for the trace of the nef b-divisor $\bM$. Under the positivity assumptions below, a klt boundary $\Gamma$ satisfies $K_X+\Gamma\sim_{\K}c(K_X+B+M_X)$ for some $c>0$. The minimal model theorem for klt pairs \cite[Theorem 1.6]{DW22}, followed by comparison of discrepancies on a common model, then gives a generalized minimal model.

\begin{theorem}\label{thm:threefold}
Let $(X,B+\bM)$ be a projective generalized klt threefold over an $F$-finite field of characteristic $p>5$. Assume that $B$ is a $\K$-divisor and $\bM$ is a nef $\K$-Cartier b-divisor, where $\K\in\{\Q,\R\}$.
\begin{enumerate}[(1)]
\item Suppose that $K_X+B+M_X$ is pseudoeffective.
If there exist a model of descent $\pi\colon Y\to X$ and $t\in\K$ with $t>-1$ such that
\begin{equation}\label{eq:shifted-positivity}
M_Y+t\pi^*(K_X+B+M_X)
\end{equation}
is semiample or nef and big, then $(X,B+\bM)$ has a minimal model.
\item Suppose that $K_X+B+M_X$ is big.
Then $(X,B+\bM)$ has a minimal model. If $\K=\Q$, then $(X,B+\bM)$ has a good minimal model and $R(X,r(K_X+B+M_X))$ is finitely generated for every sufficiently divisible integer $r>0$. If $\K=\Q$ and $K_X+B+M_X$ is nef and big, then $K_X+B+M_X$ is semiample.
\end{enumerate}
\end{theorem}

Taking $t=0$ in Theorem \ref{thm:threefold}(1) gives a minimal model when $\bM$ is b-semiample or b-big. The ground field may be imperfect, and $X$ need not be $\Q$-factorial. The boundary constructions hold over arbitrary fields under the specified resolution assumptions; the minimal model step requires the $F$-finite assumption in \cite[Theorem 1.6]{DW22}. Corollary \ref{cor:arbitrary-field-semiample} gives semiampleness over an arbitrary field with an explicit resolution assumption.

When $B$ is a $\Q$-divisor, $\bM$ is a nef $\Q$-Cartier b-divisor, and $X$ is $\Q$-factorial over a perfect field, the good minimal model assertion in Theorem \ref{thm:threefold}(2) follows from \cite[Lemma 2.6]{BF25} with $\alpha=\beta=0$. Over an algebraically closed field of characteristic $p>5$, Yang--Ye--Zhang proved termination of pseudoeffective flips and existence of log minimal models for NQC lc generalized threefolds \cite[Theorems 1.1 and 1.2]{YYZ26}. 

The rest of the paper is organized as follows. Section \ref{sec:preliminaries} fixes the notation and proves descent under separable and purely inseparable extensions. Section \ref{subsec:perturbations} proves Theorem \ref{thm:main}, and Section \ref{subsec:fixed} treats fixed parts, Theorem \ref{thm:fixed}, and anticanonical boundaries. Section \ref{subsec:generalized} is devoted to proving Theorem \ref{thm:threefold}. 

\section{Preliminaries}\label{sec:preliminaries}

\subsection{Notation}\label{subsec:notation}

A \emph{variety} is an integral separated scheme of finite type over a field. Throughout the paper, $k$ is a field and $\K\in\{\Q,\R\}$. For $\cha k=p>0$, the field $k$ is \emph{$F$-finite} if the absolute Frobenius morphism $F_k\colon\Spec k\to\Spec k$ is finite, equivalently, $[k\colon k^p]<\infty$, where $k^p\coloneqq\{a^p\mid a\in k\}$.

In characteristic $p>0$, choose a separable closure $k^{\mathrm{sep}}$ and an embedding $\overline{\F}_p\hookrightarrow k^{\mathrm{sep}}$ extending $\F_p\subset k$. Let $k_\infty\coloneqq k\overline{\F}_p =\bigcup_{r\geq1}k\F_{p^r}$ be the compositum of $k$ and $\overline{\F}_p$ in $k^{\mathrm{sep}}$. Each extension $k\F_{p^r}/k$ is finite separable, and $k\F_{p^r}\subseteq k\F_{p^s}$ whenever $r\mid s$.

For a field extension $L/k$, put $X_L\coloneqq X\times_k\Spec L$ and let $g_L\colon X_L\to X$ be the projection. Write $D_L\coloneqq g_L^*D$ when the pullback is defined. A normal scheme $Z$ with finitely many irreducible components is a disjoint union $Z=\coprod_\alpha Z_\alpha$ of integral normal schemes. Define $Q(Z)\coloneqq\prod_\alpha k(Z_\alpha)$. If $U=\Spec A\subset Z$ meets every irreducible component, then $Q(Z)$ is the total quotient ring of $A$, obtained by inverting all nonzero divisors. For $u=(u_\alpha)_\alpha\in Q(Z)^\times$, define $\divisor_Z(u)|_{Z_\alpha}\coloneqq\divisor_{Z_\alpha}(u_\alpha)$. Divisors, linear equivalence, and the singularities on $Z$ are understood componentwise.

For a finite Galois extension $L/k$, the isomorphism
\[
L\otimes_kL\simeq\prod_{\sigma\in\Gal(L/k)}L,
\qquad
a\otimes b\longmapsto(a\sigma(b))_\sigma,
\]
shows that $X_L\to X$ is a finite \'etale $\Gal(L/k)$-torsor, i.e., the natural morphism
\[
\coprod_{\sigma\in\Gal(L/k)}X_L
\longrightarrow X_L\times_X X_L,
\qquad
x\longmapsto (x,\sigma(x))
\]
on the component indexed by $\sigma$ is an isomorphism.

A \emph{prime divisor over $X$} is a prime divisor $E$ on a normal birational model $h\colon W\to X$. Write $\Cent_X E$ for the center of $E$ on $X$. For $C=\sum_P c_PP$, put $\mult_P C\coloneqq c_P$. For an $\R$-Cartier divisor $D$ on $X$, define $\ord_E D\coloneqq\mult_E(h^*D)$. For a nonzero coherent ideal $\mathfrak a\subseteq\OO_X$, the integer $\ord_E\mathfrak a$ is determined by $\mathfrak a\OO_{W,E}=(t_E^{\ord_E\mathfrak a})$, where $t_E$ is a uniformizer of the discrete valuation ring $\OO_{W,E}$, i.e., a generator of its maximal ideal. For an $\R$-Cartier divisor $D_Y$ on a normal birational model $Y\to X$, choose a normal common model $\rho\colon W'\to Y$ and $W'\to W$. Write $E'$ for the strict transform of $E$ on $W'$, and define $\ord_E D_Y\coloneqq\mult_{E'}(\rho^*D_Y)$.

For $\K$-divisors, \emph{$\K$-linear equivalence} means $D_1-D_2=\sum_j c_j\divisor_X(u_j)$, where $c_j\in\K$ and $u_j\in k(X)^\times$. A divisor of the form $\sum_j c_j\divisor_X(u_j)$ is called \emph{$\K$-principal}. A $\Q$-Cartier $\Q$-divisor $D$ is \emph{semiample} if there exists an integer $m>0$ such that $mD$ is Cartier and $\OO_X(mD)$ is generated by global sections. An $\R$-Cartier $\R$-divisor $D$ is semiample if $D\sim_{\R}\sum_j b_jL_j$, where $b_j>0$ and each $L_j$ is a semiample Cartier divisor.

Let $f\colon X\to Z$ be a projective morphism. A $\Q$-Cartier $\Q$-divisor $D$ is \emph{$f$-semiample} if there exists an integer $m>0$ such that $mD$ is Cartier and
\[
f^*f_*\OO_X(mD)\longrightarrow\OO_X(mD)
\]
is surjective. An $\R$-Cartier $\R$-divisor $D$ is $f$-semiample if $D\sim_{\R}\sum_j b_jL_j$, where $b_j>0$ and each $L_j$ is an $f$-semiample Cartier divisor. For $\K$-Cartier $\K$-divisors $D_1$ and $D_2$, write $D_1\sim_{\K,Z}D_2$ if $D_1-D_2\sim_{\K}f^*A_Z$ for a $\K$-Cartier $\K$-divisor $A_Z$ on $Z$.

Let $D$ be a $\Q$-Cartier $\Q$-divisor and let $q>0$ be an integer such that $qD$ is Cartier. The \emph{section ring} is $R(X,qD)\coloneqq\bigoplus_{\ell\geq0}H^0(X,\OO_X(\ell qD))$. The \emph{base ideal} of $|qD|$ is
\[
\mathfrak b(|qD|)\coloneqq \im\!\left(H^0(X,\OO_X(qD))\otimes_k\OO_X(-qD) \longrightarrow\OO_X\right).
\]
All global sections and complete linear systems are taken over $k$. The notation $\kappa(D)\geq0$ means that $|qD|\ne\varnothing$ for some integer $q>0$ such that $qD$ is Cartier.

A \emph{subpair} $(X,\Delta)$ consists of a normal variety $X$ and an $\R$-divisor $\Delta$ such that $K_X+\Delta$ is $\R$-Cartier. A \emph{pair} is a subpair with $\Delta\geq0$. For a proper birational morphism $h\colon W\to X$ of normal varieties, define $\Delta_W$ by $K_W+\Delta_W=h^*(K_X+\Delta)$. For a prime divisor $E$ on $W$, the \emph{log discrepancy} is $A_{X,\Delta}(E)\coloneqq1-\mult_E\Delta_W$. A subpair is \emph{sub-lc} (resp. \emph{sub-klt}), if $A_{X,\Delta}(E)\geq 0$ (resp. $A_{X,\Delta}(E)>0$) for every prime divisor $E$ over $X$. For pairs, the terms are \emph{log canonical} (lc) and \emph{klt}. A prime divisor $E$ with $A_{X,\Delta}(E)=0$ is called an \emph{lc place}. For any subpair, define the \emph{non-klt locus} by
\[
\Nklt(X,\Delta)\coloneqq
\bigcup_{A_{X,\Delta}(E)\leq0}\Cent_X E,
\]
where the union is taken over all prime divisors over $X$. If $h\colon W\to X$ is a projective log resolution, then $\Nklt(X,\Delta)=\bigcup_{\mult_E\Delta_W\geq1}h(E)$, where $E$ runs through the prime divisors on $W$. Since $h$ is proper and $\Delta_W$ has finite support, $\Nklt(X,\Delta)$ is closed. For an $\R$-Cartier divisor $D$, we have
\begin{equation}\label{eq:discrepancy-order}
A_{X,\Delta+D}(E)=A_{X,\Delta}(E)-\ord_E D.
\end{equation}

A \emph{common log resolution} of divisors $\Delta_1,\ldots,\Delta_s$ on a normal variety $X$ is a projective birational morphism $\pi\colon Y\to X$ with $Y$ regular such that the reduced divisor consisting of all $\pi$-exceptional prime divisors and the strict transforms of the prime divisors contained in $\bigcup_i\Supp\Delta_i$ is SNC. A log resolution of $(X,\Delta,\mathfrak a_1,\ldots,\mathfrak a_r)$ also satisfies $\mathfrak a_i\OO_Y=\OO_Y(-F_i)$ for effective Cartier divisors $F_i$, and the reduced divisor consisting of all $\pi$-exceptional prime divisors, the strict transforms of the prime divisors in $\Supp\Delta$, and the prime divisors in $\bigcup_i\Supp F_i$ is SNC. If $Y$ is regular and $C_Y$ has SNC support, then $(Y,C_Y)$ is sub-lc (resp.\ sub-klt) if and only if every coefficient of $C_Y$ is at most one (resp.\ less than one). In particular, if $S_Y$ is a reduced SNC divisor and $C_Y\leq S_Y$, then $A_{Y,S_Y}(E)\leq A_{Y,C_Y}(E)$ for every prime divisor $E$ over $Y$.

A common log resolution of several divisors resolves the union of the supports of the divisors. By \cite[Theorem 2.7]{DW22}, projective log resolutions exist in dimension at most three over an $F$-finite field of positive characteristic. For finitely many ideals, successive blow-ups make the pullback ideals invertible, and a log resolution then resolves the divisors defined by the pullback ideals. 

\subsection{Separable extensions and descent}\label{subsec:descent}

Separable algebraic extensions preserve sub-lc and sub-klt singularities and projective log resolutions. For a finite Galois extension, averaging conjugate divisors then gives an effective representative over the ground field.

\begin{lemma}\label{lem:separable}
Let $L/k$ be a separable algebraic extension and let $(X,\Delta)$ be a subpair. Then $X_L$ is normal, and $(X_L,\Delta_L)$ is sub-lc (resp.\ sub-klt) if and only if $(X,\Delta)$ is sub-lc (resp.\ sub-klt). For a projective log resolution $\pi\colon Y\to X$ and an irreducible component $V\subseteq X_L$, the restriction $\pi_L^{-1}(V)\to V$ of the base change $\pi_L\colon Y_L\to X_L$ is a projective log resolution of $(V,\Delta_L|_V)$.
\end{lemma}

\begin{proof}
By \cite[Tag 0C3M]{Stacks}, the scheme $X_L$ is normal. Suppose first that $L/k$ is finite. Let $h\colon W\to X$ be a normal birational morphism, let $h_L\colon W_L\to X_L$ be the base change of $h$, and let $g_W\colon W_L\to W$ be the projection. Since $L/k$ is finite separable, $g_W$ is finite \'etale. Write $K_W+\Delta_W=h^*(K_X+\Delta)$. Then we have
\[
K_{W_L}+g_W^*\Delta_W =g_W^*(K_W+\Delta_W) =h_L^*(K_{X_L}+\Delta_L).
\]
Let $F$ be a prime divisor on $W_L$ with $g_W(F)=E$, where $E$ is a prime divisor on $W$. The local homomorphism $\OO_{W,E}\to\OO_{W_L,F}$ is \'etale, and a uniformizer of $\OO_{W,E}$ has order one in $\OO_{W_L,F}$. Consequently, $\mult_F g_W^*\Delta_W=\mult_E\Delta_W$, and
\begin{equation}\label{eq:etale-discrepancy}
A_{X_L,\Delta_L}(F)
=1-\mult_F g_W^*\Delta_W
=1-\mult_E\Delta_W
=A_{X,\Delta}(E).
\end{equation}

Now, let $F$ be a prime divisor over an irreducible component $V\subseteq X_L$. The field extension $k(V)/k(X)$ is finite, and the restriction $v\coloneqq\ord_F|_{k(X)^\times}$ has residue field satisfying $\trdeg_k\kappa(v)=\dim X-1$. Indeed, the residue field of $\ord_F$ is finite over $\kappa(v)$. Choose a normal birational model $h\colon W\to X$ and a prime divisor $E\subset W$ such that $v=e\ord_E$ for an integer $e>0$. The component of $W_L$ with function field $k(V)$ is the normalization of $W$ in $k(V)$. The valuation $\ord_F$ is therefore represented by a prime divisor on $W_L$ mapping onto $E$. Since $W_L\to W$ is \'etale, $e=1$, and \eqref{eq:etale-discrepancy} holds for $F$. Every prime divisor $E$ on a normal model $W\to X$ also has a prime divisor on $W_L$ mapping onto $E$. Thus, $(X_L,\Delta_L)$ is sub-lc (resp.\ sub-klt) if and only if $(X,\Delta)$ is sub-lc (resp.\ sub-klt) when $L/k$ is finite.

Assume that $(X,\Delta)$ is sub-lc (resp.\ sub-klt). Let $L/k$ be an arbitrary separable algebraic extension, and write $L=\varinjlim_\alpha L_\alpha$ with $L_\alpha/k$ finite. Fix a prime divisor $F$ on a normal birational model $W\to V$ of an irreducible component $V\subseteq X_L$. By \cite[Tag 01ZM]{Stacks}, after enlarging a finite subextension $L_\alpha/k$ if necessary, the data $V$, $W\to V$, and $F$ descend to an irreducible component $V_\alpha\subseteq X_{L_\alpha}$, a birational morphism $h_\alpha\colon W_\alpha\to V_\alpha$, and a prime divisor $F_\alpha\subset W_\alpha$ such that $(V_\alpha)_L\simeq V$, $(W_\alpha)_L\simeq W$, and $(F_\alpha)_L\simeq F$. Since $(W_\alpha)_L\simeq W$ is normal and $(W_\alpha)_L\to W_\alpha$ is faithfully flat, \cite[Tag 033G]{Stacks} gives normality of $W_\alpha$. Let $\Delta_\alpha\coloneqq\Delta_{L_\alpha}|_{V_\alpha}$ and write $K_{W_\alpha}+C_\alpha=h_\alpha^*(K_{V_\alpha}+\Delta_\alpha)$. The separability of $L/L_\alpha$ implies that $\kappa(F_\alpha)\otimes_{L_\alpha}L$ is reduced. Hence, a uniformizer at $F_\alpha$ has order one at $F$, and by base change, we have
\[
A_{X_L,\Delta_L}(F) =1-\mult_F(C_\alpha)_L =1-\mult_{F_\alpha}C_\alpha =A_{V_\alpha,\Delta_\alpha}(F_\alpha).
\]
The finite extension case gives $(V_\alpha,\Delta_\alpha)$ sub-lc (resp.\ sub-klt). Hence, $A_{X_L,\Delta_L}(F)=A_{V_\alpha,\Delta_\alpha}(F_\alpha)\geq0$ (resp.\ $>0$). Since $F$ was arbitrary, $(X_L,\Delta_L)$ is sub-lc (resp.\ sub-klt).

For the reverse implication, let $E$ be a prime divisor on a normal birational model $h\colon W\to X$. Choose a prime divisor $F\subset W_L$ mapping onto $E$. The algebra $\kappa(E)\otimes_kL$ is reduced, and the coefficient of $F$ in $g_W^*E$ is one. The crepant equality on $W_L$ therefore gives $A_{X_L,\Delta_L}(F)=A_{X,\Delta}(E)$. Hence, $(X_L,\Delta_L)$ sub-lc (resp.\ sub-klt) implies that $(X,\Delta)$ is sub-lc (resp.\ sub-klt).

By \cite[Tags 07EQ and 07C1]{Stacks}, the morphism $\Spec(R\otimes_kL)\to\Spec R$ is regular for every finite type $k$-algebra $R$. By \cite[Tag 031E]{Stacks}, $Y_L$ and the base changes of all intersections of components of the SNC divisor on $Y$ are regular. Thus, the boundary and exceptional locus on $Y_L$ have SNC support. Let $U\subseteq X$ be a dense open subset over which $\pi$ is an isomorphism. Each component $V\subseteq X_L$ meets $U_L$, and $\pi_L$ is an isomorphism over $U_L$. Since projectivity is preserved under base change, $\pi_L^{-1}(V)\to V$ is a projective birational morphism from a regular variety with the SNC support. Therefore, $\pi_L^{-1}(V)\to V$ is a projective log resolution of $(V,\Delta_L|_V)$.
\end{proof}

The field $k_\infty$ contains $\overline{\F}_p$, which is an infinite perfect field, as required in \cite[Theorem 1]{Tan17}. The $F$-finiteness of $k_\infty$ reduces to the behavior of the imperfection degree under separable algebraic extensions.

\begin{lemma}\label{lem:constant-field}
Let $k$ be a field of characteristic $p>0$ and let $L/k$ be a separable algebraic extension. Then $L=L^pk$ and $[L\colon L^p]=[k\colon k^p]$. In particular, $k_\infty/k$ is separable algebraic and $[k_\infty\colon k_\infty^p]=[k\colon k^p]$.
\end{lemma}

\begin{proof}
The extension $L/(L^pk)$ is separable and purely inseparable, hence $L=L^pk$. The extension $L^p/k^p$ is separable algebraic, whereas $k/k^p$ is purely inseparable. Therefore, $L^p$ and $k$ are linearly disjoint over $k^p$. A $k^p$-basis of $k$ is an $L^p$-basis of $L^pk=L$, even when the basis is infinite. Thus, $[L\colon L^p]=[k\colon k^p]$. The last assertion follows from the separability of $k_\infty/k$.
\end{proof}

For a separable algebraic extension, the divisor and the rational functions defining the linear equivalence descend to a finite subextension. Taking the average of the conjugate divisors preserves the discrepancy inequalities, while the norm descends the linear equivalence to $k$.

\begin{proposition}\label{prop:galois}
Let $K/k$ be a separable algebraic field extension, let $X$ be a normal variety, and let $(X,\Delta_i)$ be subpairs for $1\leq i\leq s$. Let $D$ be a $\K$-Cartier $\K$-divisor. If $H_K\geq0$, $H_K\sim_{\K}D_K$, and every $(X_K,(\Delta_i)_K+H_K)$ is sub-lc, then there exists an effective $\K$-Cartier $\K$-divisor $H_0\sim_{\K}D$ such that every $(X,\Delta_i+H_0)$ is sub-lc. If $(X_K,(\Delta_i)_K+H_K)$ is sub-klt for an index $i$, then $(X,\Delta_i+H_0)$ is sub-klt.
\end{proposition}

\begin{proof}
Write $H_K-D_K=\sum_j c_j\divisor_{X_K}(u_j)$, where $c_j\in\K$ and $u_j\in Q(X_K)^\times$. Since $Q(X_K)=\varinjlim_{L_0}Q(X_{L_0})$, where $L_0/k$ runs through the finite subextensions of $K/k$, the functions $u_j$ are defined over one finite subextension $L_0/k$ \cite[Tag 01ZM]{Stacks}. Choose $u_{j,L_0}\in Q(X_{L_0})^\times$ with pullback $u_j$, and put $H_{L_0}\coloneqq D_{L_0}+\sum_jc_j\divisor_{X_{L_0}}(u_{j,L_0})$. The pullback of $H_{L_0}$ to $X_K$ is $H_K$. For a prime divisor $P$ on $X_{L_0}$ and a prime divisor $Q$ on $X_K$ mapping onto $P$, separability gives $\mult_P H_{L_0}=\mult_Q H_K\geq0$. Thus, $H_{L_0}\geq0$. By Lemma \ref{lem:separable}, every $(X_{L_0},(\Delta_i)_{L_0}+H_{L_0})$ is sub-lc. For every index $i$ with $(X_K,(\Delta_i)_K+H_K)$ sub-klt, the subpair $(X_{L_0},(\Delta_i)_{L_0}+H_{L_0})$ is sub-klt.

Choose a finite Galois extension $L/k$ containing $L_0$, and let $H$ be the pullback of $H_{L_0}$ to $X_L$. Write $g\colon X_L\to X$, $d\coloneqq[L\colon k]$, and $\mathcal G\coloneqq\Gal(L/k)$. The morphism $g$ is a finite \'etale $\mathcal G$-torsor. For a prime divisor $P$ on $X$, the group $\mathcal G$ acts transitively on the prime divisors $Q\subset X_L$ with $g(Q)=P$. For every prime divisor $Q\subset X_L$ with $g(Q)=P$, we have $|\Stab_{\mathcal G}(Q)|=[\kappa(Q)\colon\kappa(P)]$ and $g_*Q=[\kappa(Q)\colon\kappa(P)]P$. By comparing the coefficients of prime divisors, we obtain
\begin{equation}\label{eq:torsor-divisors}
g_*g^*D=dD,
\qquad
g^*g_*H=\sum_{\sigma\in\mathcal G}\sigma^*H.
\end{equation}

Write $X_L=\coprod_\alpha X_\alpha$ and $K_\alpha\coloneqq k(X_\alpha)$. Then $Q(X_L)=k(X)\otimes_kL=\prod_\alpha K_\alpha$. For $u=(u_\alpha)_\alpha\in Q(X_L)^\times$, define $\Nm(u)\coloneqq\prod_\alpha N_{K_\alpha/k(X)}(u_\alpha)$. The field norm $N_{K_\alpha/k(X)}(u_\alpha)$ is the determinant of multiplication by $u_\alpha$ on $K_\alpha$ as a vector space over $k(X)$. By \cite[Tag 02RT]{Stacks}, for every prime divisor $P$ on $X$,
\[
\ord_P\Nm(u) =\sum_{g(Q)=P}[\kappa(Q)\colon\kappa(P)]\ord_Q(u).
\]
The coefficient of $P$ in $g_*\divisor_{X_L}(u)$ equals $\sum_{g(Q)=P}[\kappa(Q)\colon\kappa(P)]\ord_Q(u)$, and hence, $g_*\divisor_{X_L}(u)=\divisor_X(\Nm(u))$. The equality $H-g^*D=\sum_jc_j\divisor_{X_L}(u_j)$ then gives $g_*H-dD=\sum_jc_j\divisor_X(\Nm(u_j))$. Let $H_0\coloneqq\frac{1}{d}g_*H$. Then $H_0\geq0$ and
\begin{equation}\label{eq:norm-equivalence}
H_0-D=\sum_j\frac{c_j}{d}\divisor_X\!\left(\Nm(u_j)\right).
\end{equation}
Equation \eqref{eq:norm-equivalence} shows that $H_0$ is $\K$-Cartier and $H_0\sim_{\K}D$.

For every $\sigma\in\mathcal G$ and $1\leq i\leq s$, the subpair $(X_L,g^*\Delta_i+\sigma^*H)$ is sub-lc. By \eqref{eq:torsor-divisors}, we obtain $g^*H_0=\frac{1}{d}\sum_{\sigma\in\mathcal G}\sigma^*H$. Hence, for every prime divisor $F$ over a component of $X_L$,
\begin{equation}\label{eq:average-discrepancy}
A_{X_L,g^*\Delta_i+g^*H_0}(F)
=\frac{1}{d}\sum_{\sigma\in\mathcal G}
A_{X_L,g^*\Delta_i+\sigma^*H}(F)\geq0.
\end{equation}
By Lemma \ref{lem:separable}, $(X,\Delta_i+H_0)$ is sub-lc for every $i$. If $(X_K,(\Delta_i)_K+H_K)$ is sub-klt, then every summand in \eqref{eq:average-discrepancy} is positive and $(X,\Delta_i+H_0)$ is sub-klt by Lemma \ref{lem:separable}.
\end{proof}

\subsection{Purely inseparable extensions}\label{subsec:inseparable}

For a purely inseparable extension, Tanaka's formula for canonical divisors on the normalization of the base change \cite[Proposition 2.10]{Tan18} gives an explicit relation between $A_{X,C}(E)$ and $A_{X_K,g^*C}(F)$. Since the additional terms in this relation are nonnegative, sub-lc and sub-klt singularities descend from $X_K$ to $X$. Taking $p$-powers of the finitely many rational functions defining a linear equivalence then descends an effective representative. For the projection $g\colon X_K\to X$ between normal schemes, $g^*C$ denotes the flat pullback of a Weil divisor $C$. The converse to Lemma \ref{lem:inseparable}(1) fails in general, see Example \ref{ex:inseparable-base-change}.

\begin{lemma}\label{lem:inseparable}
Let $K/k$ be a purely inseparable algebraic extension of fields of characteristic $p>0$. Let $X$ be a normal variety such that $X_K$ is normal, and let $g\colon X_K\to X$ be the projection.
\begin{enumerate}[(1)]
\item Let $(X,C)$ be a subpair.
If $(X_K,g^*C)$ is sub-lc (resp.\ sub-klt), then $(X,C)$ is sub-lc (resp.\ sub-klt).
\item Let $D$ be a $\K$-Cartier $\K$-divisor and let $(X,C_i)$ be
subpairs for $1\leq i\leq s$. If $H_K\geq0$, $H_K\sim_{\K}g^*D$, and every $(X_K,g^*C_i+H_K)$ is sub-lc, then there exists an effective $\K$-Cartier $\K$-divisor $H\sim_{\K}D$ such that $g^*H=H_K$ and every $(X,C_i+H)$ is sub-lc. If $(X_K,g^*C_i+H_K)$ is sub-klt for an index $i$, then $(X,C_i+H)$ is sub-klt.
\end{enumerate}
\end{lemma}

\begin{proof}
Since $K/k$ is purely inseparable, $X_K\to X$ induces a homeomorphism on the underlying topological spaces. Thus, $X_K$ is irreducible, and its normality implies that $X_K$ is integral. The ring $k(X)\otimes_kK$ is a field and a purely inseparable algebraic extension of $k(X)$.

For (1), let $E$ be a prime divisor on a normal birational model $h\colon W\to X$, and write $K_W+C_W=h^*(K_X+C)$. For every affine open $\Spec R\subseteq W$, the inclusion $R\subseteq k(X)$ gives $R\otimes_kK\subseteq k(X)\otimes_kK$. Hence, $W_K$ is also integral. The normalization $\nu\colon V\to W_K$ is finite since $W_K$ is of finite type over the field $K$. Let $g_W\colon W_K\to W$ be the projection and $h_K\colon W_K\to X_K$ the base change of $h$. Let $\tau\coloneqq g_W\circ\nu$ and $\rho\coloneqq h_K\circ\nu$. The morphism $\rho\colon V\to X_K$ is birational. Since $K/k$ is purely inseparable, the projection $W_K\to W$ induces a homeomorphism on the underlying topological spaces. Hence, there is a unique codimension one point of $W_K$ lying over the generic point of $E$. A prime divisor $F$ on the normalization $V$ lies over this point. For a uniformizer $t_E$ at $E$, let $e_F\coloneqq\ord_F(\tau^*t_E)>0$. Units at $E$ remain units at $F$, and $\ord_F(\tau^*u)=e_F\ord_E(u)$ for every $u\in k(X)^\times$.

Choose a regular affine neighborhood $U$ of the generic point of $E$. By \cite[Tag 0E9J]{Stacks}, a closed immersion $U\hookrightarrow\A^N_k$ is a regular immersion. By \cite[Tag 067P]{Stacks}, the base change $U_K\hookrightarrow\A^N_K$ is also a regular immersion. Thus, $U_K$ is a local complete intersection over $K$ and is Gorenstein by \cite[Tag 0BVA]{Stacks}. By \cite[Theorem 2.9]{Tan18}, $\omega_{U_K/K}\simeq\omega_{U/k}\otimes_kK$.

Let $A$ be the local ring of $U_K$ at the generic point of $(E\cap U)_K$, and let $B$ be the normalization of $A$ in $\operatorname{Frac}A$. The ring $A$ is Gorenstein and $B$ is finite over $A$. Hence, $\omega_B\simeq\Hom_A(B,\omega_A)$. Choose a generator of $\omega_A$. Under the multiplication embedding into $\operatorname{Frac}A$, $\Hom_A(B,A)$ is the conductor ideal $\mathfrak c\coloneqq\{b\in B\mid bB\subseteq A\}$. At $F$, we have $\mathfrak c\OO_{V,F}=\varpi_F^{c_F}\OO_{V,F}$, where $\varpi_F$ is a uniformizer of $\OO_{V,F}$ and $c_F\geq0$. The rational sections of the dualizing sheaves therefore give
\begin{equation}\label{eq:conductor-canonical}
K_V+\mathfrak C=\tau^*K_W,
\qquad \mathfrak C\geq0,
\qquad \mult_F\mathfrak C=c_F
\end{equation}
near the generic point of $F$, see \cite[Proposition 2.10]{Tan18}. The pullback $\tau^*K_W$ is computed on the regular open subset $U$. 

On $X_{\mathrm{reg}}$, base change of the dualizing sheaf gives $K_{X_K}=g^*K_X$. Both $X\setminus X_{\mathrm{reg}}$ and $g^{-1}(X\setminus X_{\mathrm{reg}})$ have codimension at least two. Consequently, $K_{X_K}=g^*K_X$ as Weil divisors on $X_K$. Pulling back $K_W+C_W=h^*(K_X+C)$ to $V$ and using $K_{X_K}=g^*K_X$ and \eqref{eq:conductor-canonical} gives
\[
\rho^*(K_{X_K}+g^*C)
=K_V+\tau^*C_W+\mathfrak C
\]
near $F$. The coefficient of $F$ in $\tau^*C_W+\mathfrak C$ is $e_F\mult_E C_W+c_F$. Hence,
\begin{equation}\label{eq:inseparable-discrepancy}
A_{X_K,g^*C}(F)=e_F A_{X,C}(E)-(e_F-1)-c_F,
\end{equation}
equivalently, $A_{X,C}(E)=\frac{A_{X_K,g^*C}(F)+(e_F-1)+c_F}{e_F}$. Since $e_F\geq1$ and $c_F\geq0$, $A_{X_K,g^*C}(F)\geq0$ implies $A_{X,C}(E)\geq0$, and $A_{X_K,g^*C}(F)>0$ implies $A_{X,C}(E)>0$. The same computation also holds for an infinite algebraic extension $K/k$, since $W_K$ is of finite type over $K$, the normalization $V\to W_K$ is finite, and \eqref{eq:inseparable-discrepancy} is computed at the generic point of $F$.

For (2), write $H_K-g^*D=\sum_j b_j\divisor_{X_K}(u_j)$ with $b_j\in\K$ and $u_j\in k(X_K)^\times$. The extension $k(X_K)/k(X)$ is purely inseparable algebraic. Since the set of functions $u_j$ is finite, there exists a number $q=p^a$ such that $v_j\coloneqq u_j^q\in k(X)^\times$ for every $j$. Define
\begin{equation}\label{eq:inseparable-representative}
H\coloneqq D+\sum_j\frac{b_j}{q}\divisor_X(v_j).
\end{equation}
The divisor $H$ in \eqref{eq:inseparable-representative} is $\K$-Cartier and satisfies $H\sim_{\K}D$. Since $g^*\divisor_X(v_j)=q\divisor_{X_K}(u_j)$ for every $j$, $g^*H=H_K$. For every prime divisor $P\subset X$, let $Q\subset X_K$ be the prime divisor mapping onto $P$ and let $e_Q\coloneqq\mult_Qg^*P>0$. Then $\mult_QH_K=e_Q\mult_PH$, and $H_K\geq0$ implies $H\geq0$. Part (1) gives $(X,C_i+H)$ sub-lc for every $i$, and sub-klt for each $i$ for which $(X_K,g^*C_i+H_K)$ is sub-klt.
\end{proof}

\begin{example}\label{ex:inseparable-base-change}
Lc singularities need not be preserved under purely inseparable base change. Let $a\in k\setminus k^p$, let $P\subset\A^1_k$ be the closed point defined by $x^p-a$, and put $K=k(\sqrt[p]{a})$. The pair $(\A^1_k,P)$ is lc, whereas $P_K=p[\sqrt[p]{a}]$ and $(\A^1_K,P_K)$ is not lc.
\end{example}

We will repeatedly use the following birational observation.
Let $\pi\colon Y\to X$ be a proper birational morphism of normal varieties. If $D$ is $\K$-Cartier and $H=\pi^*D+\sum_jc_j\divisor_Y(u_j)$, then
\begin{equation}\label{eq:birational-pullback}
\pi_*H=D+\sum_jc_j\divisor_X(u_j),
\qquad \pi^*(\pi_*H)=H.
\end{equation}
In particular, $\pi_*H$ is $\K$-Cartier and $\pi_*H\sim_{\K}D$. If $H\geq0$, then $\pi_*H\geq0$.

\section{Main results and proofs}\label{sec:results}

In this section, we prove the main results.
\subsection{Semiample perturbations}\label{subsec:perturbations}

The proof first applies Tanaka's semiample perturbation theorem over an $F$-finite field containing an infinite perfect subfield \cite[Theorem 1]{Tan17}. We then descend an effective divisor in the given linear equivalence class along separable extensions by averaging its conjugates and along purely inseparable extensions by taking suitable $p$-powers of the rational functions defining the linear equivalence. Perturbing a sufficiently divisible multiple of $\pi^*D$ on a log resolution and dividing by that multiple gives the inequalities in Theorem \ref{thm:main}.

\begin{proposition}\label{prop:regular}
Let $Y$ be a regular projective variety over a field $k$, and let $S_Y$ be a reduced SNC divisor. For every semiample $\K$-Cartier $\K$-divisor $P_Y$, there exists an effective divisor $H\sim_{\K}P_Y$ such that $(Y,S_Y+H)$ is lc.
\end{proposition}

\begin{proof}
Suppose first that $\cha k=0$. Write $P_Y\sim_{\K}\sum_jb_jL_j$, where $b_j\in\K_{>0}$ and the $L_j$ are semiample Cartier divisors. Choose integers $q_j>0$ such that $|q_jL_j|$ is base point free and $\frac{b_j}{q_j}\leq1$. Bertini's theorem over the infinite field $k$ gives divisors $T_j\in|q_jL_j|$ such that $S_Y+\sum_jT_j$ is reduced SNC. Let $H\coloneqq\sum_j\frac{b_j}{q_j}T_j$. Then $H\geq0$ and $H\sim_{\K}P_Y$. Since every coefficient of $S_Y+H$ is at most one, $(Y,S_Y+H)$ is lc.

Now, suppose that $\cha k=p>0$ and that $k$ is $F$-finite. By Lemmas \ref{lem:separable} and \ref{lem:constant-field}, the field $k_\infty$ is $F$-finite and $Y_{k_\infty}=\coprod_\alpha Y_\alpha$ is a disjoint union of regular projective varieties. Put $S_\alpha\coloneqq(S_Y)_{k_\infty}|_{Y_\alpha}$. Each $S_\alpha$ is reduced SNC. By \cite[Theorem 1]{Tan17}, there is an effective divisor $H_\alpha\sim_{\K}(P_Y)_{k_\infty}|_{Y_\alpha}$ such that $(Y_\alpha,S_\alpha+H_\alpha)$ is lc. Define $H_{k_\infty}$ by $H_{k_\infty}|_{Y_\alpha}=H_\alpha$. Then $H_{k_\infty}\geq0$, $H_{k_\infty}\sim_{\K}(P_Y)_{k_\infty}$, and $(Y_{k_\infty},(S_Y)_{k_\infty}+H_{k_\infty})$ is lc. By Proposition \ref{prop:galois}, there is an effective divisor $H\sim_{\K}P_Y$ on $Y$ such that $(Y,S_Y+H)$ is lc.

From now on, let $k$ be an arbitrary field of characteristic $p>0$. Choose a $p$-basis $\Lambda$ of $k$ and write $S_Y=\sum_{i=1}^rS_i$. For every nonempty subset $I\subseteq\{1,\ldots,r\}$, let $S_I\coloneqq\bigcap_{i\in I}S_i$ be the scheme theoretic intersection. By \cite[Construction 3.1, Lemma 3.2(i), and Corollary 3.7(i)]{Mur21}, there is a cofinite subset $\Gamma\subseteq\Lambda$ such that
\[
K\coloneqq k^\Gamma
=\bigcup_{a\geq0}k\bigl(\sqrt[p^a]{\lambda}\mid\lambda\in\Gamma\bigr)
\]
is $F$-finite and $Y_K$ and every $(S_I)_K$ are regular. The cofinite subset $\Gamma$ is chosen simultaneously for $Y$, all $S_i$, and all nonempty intersections $S_I$. Each $(S_i)_K$ is reduced, and each nonempty $(S_I)_K$ has codimension $|I|$ in $Y_K$. Hence, $(S_Y)_K$ is reduced SNC. The $F$-finite case gives $H_K\geq0$ with $H_K\sim_{\K}(P_Y)_K$ and $(Y_K,(S_Y)_K+H_K)$ lc. Since $K/k$ is purely inseparable algebraic, Lemma \ref{lem:inseparable}(2) gives an effective divisor $H\sim_{\K}P_Y$ on $Y$ such that $(Y,S_Y+H)$ is lc.
\end{proof}

\begin{remark}\label{rem:order-bound}
Let $Y$ be a regular projective variety, let $S_Y$ be a reduced SNC
divisor, and let $P_Y$ be a semiample $\K$-Cartier $\K$-divisor.
Let $(Y,C_Y)$ be a subpair with $C_Y\leq S_Y$.
Fix $0<\epsilon<1$ and choose an integer
$n\geq\frac{1}{\epsilon}$.
By Proposition \ref{prop:regular}, there is an effective divisor
$H\sim_{\K}nP_Y$ such that $(Y,S_Y+H)$ is lc.
Put $H_\epsilon\coloneqq\frac{1}{n}H$.
Then
\begin{equation}\label{eq:order-bound}
0\leq\ord_E H_\epsilon
\leq\frac{1}{n}A_{Y,S_Y}(E)
\leq\epsilon A_{Y,C_Y}(E)
\end{equation}
for every prime divisor $E$ over $Y$.
In particular, $H_\epsilon$ may be chosen independently of $C_Y$.
\end{remark}

\begin{proof}[Proof of Theorem \ref{thm:main}]
Choose a projective log resolution $\pi\colon Y\to X$ and write $K_Y+\Delta_Y=\pi^*(K_X+\Delta)$. Let $S_Y$ be the reduced SNC divisor consisting of all $\pi$-exceptional prime divisors and the strict transforms of the prime divisors in $\Supp\Delta$. Since $(X,\Delta)$ is sub-lc, $\Delta_Y\leq S_Y$. Choose an effective divisor $H_\epsilon\sim_{\K}\pi^*D$ satisfying \eqref{eq:order-bound} with $C_Y=\Delta_Y$. Then $\ord_E H_\epsilon\leq\epsilon A_{X,\Delta}(E)$ for every $E$. Write $H_\epsilon=\pi^*D+\sum_jc_j\divisor_Y(u_j)$ and put $D_\epsilon\coloneqq\pi_*H_\epsilon$. Pushing forward the principal divisors gives $D_\epsilon=D+\sum_jc_j\divisor_X(u_j)$, and pulling back gives $\pi^*D_\epsilon=H_\epsilon$. Thus, $D_\epsilon$ is effective and $D_\epsilon\sim_{\K}D$. The equality $K_Y+\Delta_Y+H_\epsilon=\pi^*(K_X+\Delta+D_\epsilon)$ and \eqref{eq:discrepancy-order} give $A_{X,\Delta+D_\epsilon}(E)=A_{X,\Delta}(E)-\ord_E H_\epsilon$. Together with $0\leq\ord_E H_\epsilon\leq\epsilon A_{X,\Delta}(E)$, we obtain \eqref{eq:main-bound}.
\end{proof}

For every prime divisor $E$ over $X$, \eqref{eq:main-bound} gives $A_{X,\Delta}(E)=0$ if and only if $A_{X,\Delta+D_\epsilon}(E)=0$. Therefore, $\Nklt(X,\Delta+D_\epsilon)=\Nklt(X,\Delta)$. The locus $\Nklt(X,\Delta)$ is closed since $(X,\Delta)$ admits a projective log resolution.

\begin{example}\label{ex:integral-representative}
An effective integral divisor $D'\in|D|$ need not preserve log canonicity, even when Theorem \ref{thm:main} gives an effective $\Q$-divisor in the class of $D$. Let $X=\P^1_{\F_2}$, $\Delta=[0]+[1]+[\infty]$, and $D=[\infty]$. Every effective integral divisor $D'\in|D|$ makes $(X,\Delta+D')$ non-lc. If $P$ is the closed point defined by $t^2+t+1$, then $\divisor_X(t^2+t+1)=P-2[\infty]$, and $D'\coloneqq\frac{1}{2}P\sim_{\Q}D$ preserves the lc places. 
\end{example}

\subsection{Fixed parts and effective boundaries}\label{subsec:fixed}

For a decomposition $\pi^*D=P_Y+N_Y$ with $P_Y$ semiample and $N_Y\geq0$, the perturbation of $P_Y$ can be chosen with $N_Y$ fixed. Writing $K_Y+\Delta_Y=\pi^*(K_X+\Delta)$ gives $A_{Y,\Delta_Y+N_Y}(E)=A_{X,\Delta}(E)-\ord_E N_Y$.

\begin{proposition}\label{prop:decomposition}
Let $(X,\Delta)$ be a projective subpair over a field $k$, and let $D$ be a $\K$-Cartier $\K$-divisor. Let $\pi\colon Y\to X$ be a projective birational morphism of normal varieties, and write $K_Y+\Delta_Y=\pi^*(K_X+\Delta)$ and $\pi^*D=P_Y+N_Y$, where $P_Y$ and $N_Y$ are $\K$-Cartier $\K$-divisors, $P_Y$ is semiample, and $N_Y\geq0$. Assume that $(Y,\Delta_Y+N_Y)$ is sub-lc and admits a projective log resolution. For every $0<\epsilon<1$, there exists an effective divisor $G_\epsilon\sim_{\K}D$ such that
\begin{equation}\label{eq:fixed-bound}
(1-\epsilon)A_{Y,\Delta_Y+N_Y}(E)
\leq A_{X,\Delta+G_\epsilon}(E)
\leq A_{Y,\Delta_Y+N_Y}(E)
\end{equation}
for every prime divisor $E$ over $X$.
\end{proposition}

\begin{proof}
By Theorem \ref{thm:main}, there exists an effective divisor $H_\epsilon\sim_{\K}P_Y$ such that $\ord_E H_\epsilon\leq\epsilon A_{Y,\Delta_Y+N_Y}(E)$ for every $E$. Set $G_\epsilon\coloneqq\pi_*(N_Y+H_\epsilon)$ and write $N_Y+H_\epsilon=\pi^*D+\sum_jc_j\divisor_Y(u_j)$. Then $G_\epsilon=D+\sum_jc_j\divisor_X(u_j)$ and $\pi^*G_\epsilon=N_Y+H_\epsilon$. In particular, $G_\epsilon\geq0$ and $G_\epsilon\sim_{\K}D$. Since $A_{X,\Delta+G_\epsilon}(E) =A_{Y,\Delta_Y+N_Y}(E)-\ord_E H_\epsilon$, \eqref{eq:fixed-bound} follows.
\end{proof}

\subsubsection{Finitely generated section rings}\label{subsubsec:finite-generation}

If $R(X,qD)$ is finitely generated, then $R(X,mqD)$ is generated in degree one for every sufficiently divisible integer $m>0$ \cite[Tag 0EGH]{Stacks}. Theorem \ref{thm:fixed} assumes a projective log resolution of the base ideal for a multiple with this property. On the resolution, the fixed part computes every $v_E(\|D\|)$, and Proposition \ref{prop:decomposition} perturbs the semiample part.

\begin{proof}[Proof of Theorem \ref{thm:fixed}]
Put $\mathfrak b\coloneqq\mathfrak b(|qD|)$ and choose the projective log resolution $\pi\colon Y\to X$ in the assumptions. Write $\mathfrak b\OO_Y=\OO_Y(-F_Y)$ and $q\pi^*D=L_Y+F_Y$. The pullbacks of the sections of $\OO_X(qD)$ generate $\OO_Y(L_Y)$. Since $R(X,qD)$ is generated in degree one, every section of $\OO_X(\ell qD)$ is a linear combination of products of $\ell$ sections of $\OO_X(qD)$. Hence,
\begin{equation}\label{eq:fixed-systems}
\mathfrak b(|\ell qD|)=\mathfrak b^\ell,
\qquad |\ell q\pi^*D|=|\ell L_Y|+\ell F_Y
\quad(\ell\geq1).
\end{equation}
Moreover, every section of $\OO_Y(\ell q\pi^*D)$ vanishes along $\ell F_Y$. Multiplication by the section defining $\ell F_Y$ therefore gives an isomorphism $H^0(Y,\OO_Y(\ell L_Y))\simeq H^0(Y,\OO_Y(\ell q\pi^*D))$. Let $P_Y\coloneqq\frac{1}{q}L_Y$ and $N_Y\coloneqq\frac{1}{q}F_Y$.

Let $G$ be an effective divisor such that $G\sim_{\Q}D$. Choose an integer $n>0$ divisible by $q$ such that $nG\in|nD|$. The equality of linear systems in \eqref{eq:fixed-systems} gives $\pi^*G\geq N_Y$ and hence $\ord_E G\geq\ord_E N_Y$ for every $E$.

Fix a prime divisor $E$ over $X$, and let $\eta_E$ be the generic point of $\Cent_Y E$. Since $\OO_Y(qP_Y)$ is generated by global sections, there exists $s_E\in H^0(Y,\OO_Y(qP_Y))$ that does not vanish at $\eta_E$. The zero divisor $Q_E$ of $s_E$ satisfies $\ord_E Q_E=0$. Set $G_E\coloneqq\pi_*(N_Y+\frac{1}{q}Q_E)$. The linear equivalence $N_Y+\frac{1}{q}Q_E\sim_{\Q}\pi^*D$ and \eqref{eq:birational-pullback} give $G_E\sim_{\Q}D$ and $\pi^*G_E=N_Y+\frac{1}{q}Q_E$. Thus, $G_E\geq0$ and $\ord_E G_E=\ord_E N_Y$. Consequently,
\begin{equation}\label{eq:fixed-order}
v_E(\|D\|)=\ord_E N_Y=\frac{1}{q}\ord_E F_Y
\qquad\text{for every prime divisor }E\text{ over }X.
\end{equation}

If $(X,\Delta+G)$ is sub-lc, then $A_{X,\Delta}(E)\geq\ord_E G\geq v_E(\|D\|)$ for every $E$. If $(X,\Delta+G)$ is sub-klt, then $A_{X,\Delta}(E)>\ord_E G\geq v_E(\|D\|)$ for every $E$.

For the converse, the inequalities in \eqref{eq:criterion} and the identity \eqref{eq:fixed-order} imply that $(Y,\Delta_Y+N_Y)$ is sub-lc. Proposition \ref{prop:decomposition} gives an effective divisor $G_\epsilon\sim_{\Q}D$ with $(X,\Delta+G_\epsilon)$ sub-lc. If all inequalities in \eqref{eq:criterion} are strict, then $A_{Y,\Delta_Y+N_Y}(E)>0$ for every $E$, and \eqref{eq:fixed-bound} gives $A_{X,\Delta+G_\epsilon}(E)\geq (1-\epsilon)A_{Y,\Delta_Y+N_Y}(E)>0$. Therefore, $(X,\Delta+G_\epsilon)$ is sub-klt.
\end{proof}

\begin{example}\label{ex:finite-generation}
The lc criterion in Theorem \ref{thm:fixed} can fail without finite generation assumption. Consider Zariski's example of a nef and big divisor that is not semiample \cite[Example 4.3]{Koi15}. Let $C_0\subset\P^2_{\C}$ be a smooth cubic, and choose twelve distinct points $p_1,\ldots,p_{12}\in C_0$ such that $\OO_{C_0}(4)(-\sum_i p_i)$ is not torsion in $\Pic^0(C_0)$. Let $X\to\P^2$ be the blow-up of $p_1,\ldots,p_{12}$, let $H$ be the pullback of a line, and let $E_i$ be the exceptional divisor over $p_i$. Write $C$ for the strict transform of $C_0$ and put $D\coloneqq4H-\sum_iE_i=H+C$. Then $D^2=4$, $C^2=-3$, $D\cdot C=0$, and $D\cdot E_i=1$. For an irreducible plane curve $T_0\ne C_0$ of degree $d$, put $m_i\coloneqq\mult_{p_i}T_0$ and let $T$ be the strict transform of $T_0$. B\'ezout's theorem gives $\sum_i m_i\leq3d$ and $D\cdot T=4d-\sum_i m_i\geq d$. For every rational number $0<\delta<1$,
\[
(D-\delta C)\cdot C=3\delta,\qquad (D-\delta C)\cdot E_i=1-\delta,\qquad (D-\delta C)^2=4-3\delta^2>0.
\]
Moreover, $(D-\delta C)\cdot T=(4-3\delta)d-(1-\delta)\sum_i m_i\geq d$. By the Nakai--Moishezon criterion, $D-\delta C$ is ample. In particular, $D$ is nef and big.

Since $\OO_C(D)$ is a nontorsion line bundle of degree zero, $H^0(C,\OO_C(mD))=0$ for every integer $m>0$. Every divisor in $|mD|$ therefore contains $C$. Consequently, every effective divisor $G\sim_{\Q}D$ satisfies $\ord_C G>0$, and $(X,C+G)$ is not lc.

Fix a prime divisor $E$ over $X$ and a rational number $0<\delta<1$. Choose an integer $m>0$ such that $m(D-\delta C)$ is Cartier and $\OO_X(m(D-\delta C))$ is generated by global sections. A section nonvanishing at the generic point of $\Cent_X E$ gives an effective divisor $Q_\delta\sim_{\Q}D-\delta C$ with $\ord_E Q_\delta=0$. Then $Q_\delta+\delta C\sim_{\Q}D$ is effective and $\ord_E(Q_\delta+\delta C)=\delta\ord_E C$. Letting $\delta$ tend to zero gives $v_E(\|D\|)=0$ for every $E$. Thus, the lc pair $(X,C)$ satisfies \eqref{eq:criterion}, although no effective divisor $G\sim_{\Q}D$ makes $(X,C+G)$ lc.
\end{example}

\subsubsection{Anticanonical boundaries}\label{subsubsec:anticanonical}

Proposition \ref{prop:decomposition} constructs anticanonical boundaries over the ground field. Let $(X,\Delta)$ be a projective lc pair over $k$, with $\Delta$ a $\Q$-divisor. Let $\pi\colon Y\to X$ be a projective birational morphism of normal varieties. Assume that $-\pi^*(K_X+\Delta)=P_Y+N_Y$ satisfies the assumptions of Proposition \ref{prop:decomposition} with $\K=\Q$. For the divisor $G_\epsilon$ given by that proposition, put $\Gamma_\epsilon\coloneqq\Delta+G_\epsilon$. Then $\Gamma_\epsilon\geq\Delta$, $(X,\Gamma_\epsilon)$ is lc, and $K_X+\Gamma_\epsilon\sim_{\Q}0$. The lc places of $(X,\Gamma_\epsilon)$ are exactly the prime divisors $E$ satisfying $A_{Y,\Delta_Y+N_Y}(E)=0$.

Over $\C$, in the big anticanonical case, Choi--Park constructed effective anticanonical representatives realizing the potentially non-klt locus \cite[Proposition 1.4(1)]{CP16} and characterized klt Calabi--Yau type through the $s$-decomposition \cite[Theorem 5.10]{CP16}. For a decomposition with semiample positive part, the boundary construction over $\C$ also follows from Bertini's theorem. The association of a generalized pair to a potential triple with a birational Zariski decomposition appears in \cite[proof of Theorem 1.1]{CJL25}.

\begin{remark}\label{rem:potential-discrepancy}
Over $\C$, the \emph{potential log discrepancy} is defined as $\overline A_{X,\Delta}(E)\coloneqq A_{X,\Delta}(E)-\sigma_E(-(K_X+\Delta))$ \cite[Definition 3.1]{CP16}. If $-\pi^*(K_X+\Delta)=P_Y+N_Y$ is the divisorial Zariski decomposition with $P_Y$ nef, then
\[
\sigma_E(-(K_X+\Delta))=\ord_E N_Y, \qquad \overline A_{X,\Delta}(E)=A_{Y,\Delta_Y+N_Y}(E).
\]
For a decomposition with an arbitrary effective summand $N_Y$, Proposition \ref{prop:decomposition} compares the log discrepancies with $A_{Y,\Delta_Y+N_Y}(E)$, rather than with the potential log discrepancies.
\end{remark}

Let $(X,\Delta)$ be a projective lc pair with $\Delta$ a $\Q$-divisor, and let $\phi\colon X\dashrightarrow X'$ be a birational contraction. Put $\Delta'\coloneqq\phi_*\Delta$. Assume that $(X',\Delta')$ is lc and $-(K_{X'}+\Delta')$ is semiample. Let $p\colon W\to X$ and $q\colon W\to X'$ be a projective common log resolution satisfying
\[
-p^*(K_X+\Delta)=-q^*(K_{X'}+\Delta')+F_W, \qquad K_W+\Delta_W=p^*(K_X+\Delta),
\]
where $F_W\geq0$ is $q$-exceptional and $\Supp\Delta_W\cup\Supp F_W$ is SNC. Then $K_W+\Delta_W+F_W=q^*(K_{X'}+\Delta')$ and $A_{W,\Delta_W+F_W}(E)=A_{X',\Delta'}(E)$. Proposition \ref{prop:decomposition}, with $P_W=-q^*(K_{X'}+\Delta')$ and $N_W=F_W$, gives an lc boundary $\Gamma_\epsilon\geq\Delta$ satisfying $K_X+\Gamma_\epsilon\sim_{\Q}0$. The lc places of $(X,\Gamma_\epsilon)$ are exactly the lc places of $(X',\Delta')$. For $\Delta=0$, the decomposition of $-p^*K_X$ therefore gives an lc boundary $\Gamma_\epsilon$ with $K_X+\Gamma_\epsilon\sim_{\Q}0$.

\subsection{Generalized pairs and minimal models}\label{subsec:generalized}

For generalized klt pairs over $\C$, \cite[Lemma 3.4]{LX23} replaces a big nef part by an ordinary klt boundary. In positive characteristic and dimension at most three, \cite[Lemma 2.4]{BF25} constructs a klt boundary for rational generalized klt data after adding an ample $\Q$-divisor. The constructions below hold over arbitrary fields under the specified resolution assumptions.

A generalized pair $(X,B+\bM)$ consists of a normal variety $X$, an effective $\R$-divisor $B$, a projective birational morphism $\pi\colon Y\to X$ of normal varieties, and a nef $\R$-Cartier divisor $M_Y$ such that $K_X+B+M_X$ is $\R$-Cartier, where $M_X=\pi_*M_Y$. The divisor $M_Y$ determines the b-divisor $\bM$, with $M_W=\mu^*M_Y$ on every model $\mu\colon W\to Y$. The b-divisor $\bM$ is said to \emph{descend to $Y$}. Define $B_Y$ by
\begin{equation}\label{eq:generalized-crepant}
K_Y+B_Y+M_Y=\pi^*(K_X+B+M_X).
\end{equation}
For a normal model $\mu\colon W\to Y$, pulling back \eqref{eq:generalized-crepant} via $\mu$, write $K_W+B_W=\mu^*(K_Y+B_Y)$. For a prime divisor $E$ on $W$, the \emph{generalized log discrepancy} is $a_E(X,B,\bM)\coloneqq1-\mult_E B_W=A_{Y,B_Y}(E)$. The generalized pair is \emph{generalized lc} (resp. \emph{generalized klt}) if $a_E(X,B,\bM)\geq$ (resp. $a_E(X,B,\bM)>0$) for every prime divisor $E$ over $X$. A \emph{nef $\K$-Cartier b-divisor} descends to a nef $\K$-Cartier $\K$-divisor on a projective birational model.

The b-divisor $\bM$ is \emph{b-semiample} (resp. \emph{b-big}) if $\bM$ descends to a semiample (resp. big) divisor. When $B$ is a $\Q$-divisor and $\bM$ is a $\Q$-Cartier b-divisor, $K_X+B+M_X$ is a $\Q$-Cartier divisor.

A nef b-divisor is \emph{NQC} if $\bM=\sum_{j=1}^m r_j\mathbf M_j$ for finitely many nef Cartier b-divisors $\mathbf M_j$ and real numbers $r_j\geq0$ \cite[Definition 2.3]{YYZ26}. A nef $\Q$-Cartier b-divisor is NQC. 

A \emph{minimal model} is a birational contraction $\phi\colon X\dashrightarrow X'$ to a normal projective variety $X'$ such that $X'$ is $\Q$-factorial, $B'=\phi_*B$, $(X',B'+\bM)$ is generalized klt, $K_{X'}+B'+M_{X'}$ is nef, and $a_E(X',B',\bM)\geq a_E(X,B,\bM)$ for every prime divisor $E$ over $X$. The discrepancy inequality is strict for every prime divisor on $X$ contracted by $\phi$. A \emph{good minimal model} is a minimal model with $K_{X'}+B'+M_{X'}$ semiample.

Proposition \ref{prop:absorption} constructs effective boundaries from a decomposition $M_Y=P_Y+N_Y$ without assuming that $M_Y$ is nef.

\begin{proposition}\label{prop:absorption}
Let $\pi\colon Y\to X$ be a projective birational morphism of normal projective varieties over a field $k$. Let $B$ be an $\R$-divisor on $X$ and write $M_Y=P_Y+N_Y$, where $P_Y$ and $N_Y$ are $\K$-Cartier $\K$-divisors, $P_Y$ is semiample, and $N_Y\geq0$. Put $M_X\coloneqq\pi_*M_Y$ and assume that $K_X+B+M_X$ is $\R$-Cartier. Define $B_Y$ by $K_Y+B_Y+M_Y=\pi^*(K_X+B+M_X)$. Assume that $(Y,B_Y+N_Y)$ is sub-lc and admits a projective log resolution. For every $0<\epsilon<1$, there exists a sub-lc subpair $(X,\Gamma_\epsilon)$ such that $\Gamma_\epsilon\geq B$, $\Gamma_\epsilon-B$ is a $\K$-divisor, $\Gamma_\epsilon-B\sim_{\K}M_X$, and
\begin{equation}\label{eq:absorption-bound}
(1-\epsilon)A_{Y,B_Y+N_Y}(E) \leq A_{X,\Gamma_\epsilon}(E) \leq A_{Y,B_Y+N_Y}(E)
\end{equation}
for every prime divisor $E$ over $X$. If $B\geq0$, then $(X,\Gamma_\epsilon)$ is an lc pair.
\end{proposition}

\begin{proof}
By Theorem \ref{thm:main}, there exists an effective divisor $H_\epsilon\sim_{\K}P_Y$ such that $\ord_E H_\epsilon\leq\epsilon A_{Y,B_Y+N_Y}(E)$ for every $E$. Write $N_Y+H_\epsilon=M_Y+\sum_jc_j\divisor_Y(u_j)$ and put $\Gamma_\epsilon\coloneqq B+\pi_*(N_Y+H_\epsilon)$. Then $\Gamma_\epsilon\geq B$ and
\begin{align*}
K_X+\Gamma_\epsilon
&=K_X+B+M_X+\sum_jc_j\divisor_X(u_j),\\
\pi^*(K_X+\Gamma_\epsilon)&=K_Y+B_Y+N_Y+H_\epsilon.
\end{align*}
In particular, $K_X+\Gamma_\epsilon$ is $\R$-Cartier and $\Gamma_\epsilon-B\sim_{\K}M_X$. Taking coefficients in the crepant equality gives $A_{X,\Gamma_\epsilon}(E) =A_{Y,B_Y+N_Y}(E)-\ord_E H_\epsilon$. The inequalities in \eqref{eq:absorption-bound} follow.
\end{proof}

Let $(X,B+\bM)$ be a generalized lc pair with $\bM$ descending to a semiample divisor $M_Y$ on a normal projective model $\pi\colon Y\to X$. If $(Y,B_Y)$ admits a projective log resolution, then Proposition \ref{prop:absorption}, with $P_Y=M_Y$ and $N_Y=0$, gives
\begin{equation}\label{eq:generalized-absorption}
(1-\epsilon)a_E(X,B,\bM)
\leq A_{X,\Gamma_\epsilon}(E)\leq a_E(X,B,\bM).
\end{equation}
The inequalities \eqref{eq:generalized-absorption} show that a prime divisor is an lc place of $(X,\Gamma_\epsilon)$ exactly when the prime divisor is a generalized lc place of $(X,B+\bM)$; in particular, generalized klt singularities give a klt pair $(X,\Gamma_\epsilon)$. 

\begin{lemma}\label{lem:big-replacement}
Let $(X,B+\bM)$ be a projective generalized klt pair over a field $k$, where $B$ is a $\K$-divisor and $\bM$ is a nef $\K$-Cartier b-divisor. Let $\pi\colon Y\to X$ be a normal projective model of descent.
\begin{enumerate}[(1)]
\item Suppose that $M_Y\sim_{\K}A_Y+F_Y$, where $A_Y$ is ample,
$F_Y\geq0$, and $A_Y$ and $F_Y$ are $\K$-Cartier $\K$-divisors. Assume that $B_Y$ and $F_Y$ admit a common projective log resolution. Then there exists a $\K$-divisor $\Gamma\geq B$ such that $(X,\Gamma)$ is klt and $K_X+\Gamma\sim_{\K}K_X+B+M_X$.
\item Suppose that $\pi^*(K_X+B+M_X)\sim_{\K}A_Y+F_Y$, where
$A_Y$ is ample, $F_Y\geq0$, and $A_Y$ and $F_Y$ are $\K$-Cartier $\K$-divisors. Assume that $B_Y$ and $F_Y$ admit a common projective log resolution. For every $\eta>0$, there exist a rational number $1<c<1+\eta$ and a $\K$-divisor $\Gamma\geq B$ such that $(X,\Gamma)$ is klt and $K_X+\Gamma\sim_{\K}c(K_X+B+M_X)$.
\end{enumerate}
\end{lemma}

\begin{proof}
By assumption, there exists a common projective log resolution $\mu\colon V\to Y$. Write $K_V+B_V=\mu^*(K_Y+B_Y)$ and $\mu^*F_Y=\sum_i f_iE_i$. Since $(X,B+\bM)$ is generalized klt, $b_i\coloneqq\mult_{E_i}B_V<1$ for every $i$. Choose a rational number $0<\delta<1$ such that $\delta<\frac{1-b_i}{f_i}$ whenever $f_i>0$. For (2), we additionally require $\delta<\eta$. The divisor $B_V+\delta\mu^*F_Y$ has SNC support and every coefficient is less than one. Hence, $(Y,B_Y+\delta F_Y)$ is sub-klt.

For (1), the divisor $(1-\delta)M_Y+\delta A_Y$ is ample. By Theorem \ref{thm:main}, there exists an effective divisor $H_Y\sim_{\K}(1-\delta)M_Y+\delta A_Y$ such that $(Y,B_Y+\delta F_Y+H_Y)$ is sub-klt. Write $\delta F_Y+H_Y=M_Y+\Psi_Y$ with $\Psi_Y$ $\K$-principal. Set $\Gamma\coloneqq B+\pi_*(\delta F_Y+H_Y)$ and $\Psi_X\coloneqq\pi_*\Psi_Y$. Then $\pi^*\Psi_X=\Psi_Y$ and
\begin{align*}
K_X+\Gamma=K_X+B+M_X+\Psi_X,\quad \pi^*(K_X+\Gamma)=K_Y+B_Y+\delta F_Y+H_Y.
\end{align*}
Therefore, $(X,\Gamma)$ is klt, $\Gamma\geq B$, and $K_X+\Gamma\sim_{\K}K_X+B+M_X$.

For (2), the divisor $M_Y+\delta A_Y$ is ample. By Theorem \ref{thm:main}, there exists an effective divisor $H_Y\sim_{\K}M_Y+\delta A_Y$ such that $(Y,B_Y+\delta F_Y+H_Y)$ is sub-klt. There is a $\K$-principal divisor $\Psi_Y$ satisfying
\begin{equation}\label{eq:big-principal}
\delta F_Y+H_Y
=M_Y+\delta\pi^*(K_X+B+M_X)+\Psi_Y.
\end{equation}
Set $\Gamma\coloneqq B+\pi_*(\delta F_Y+H_Y)$, $\Psi_X\coloneqq\pi_*\Psi_Y$, and $c\coloneqq1+\delta$. Then we have
\begin{align}
K_X+\Gamma&=c(K_X+B+M_X)+\Psi_X,\label{eq:big-boundary}\\
\pi^*(K_X+\Gamma)&=K_Y+B_Y+\delta F_Y+H_Y.\label{eq:big-crepant}
\end{align}
Since $(Y,B_Y+\delta F_Y+H_Y)$ is sub-klt, \eqref{eq:big-crepant} gives $(X,\Gamma)$ klt. Equation \eqref{eq:big-boundary} gives $K_X+\Gamma\sim_{\K}c(K_X+B+M_X)$. The construction gives $\Gamma\geq B$ and $1<c<1+\eta$.
\end{proof}

For threefolds over an $F$-finite field of positive characteristic, the common projective log resolutions in Lemma \ref{lem:big-replacement} exist by \cite[Theorem 2.7]{DW22}.

\begin{proof}[Proof of Theorem \ref{thm:threefold}]
For (1), set $T_Y\coloneqq M_Y+t\pi^*(K_X+B+M_X)$. The b-divisor $\bM_t$ determined by $T_Y$ is nef, and
\begin{equation}\label{eq:shifted-crepant}
K_Y+B_Y+T_Y=(1+t)\pi^*(K_X+B+M_X).
\end{equation}
Pushing forward \eqref{eq:shifted-crepant} gives $K_X+B+(\bM_t)_X=(1+t)(K_X+B+M_X)$. The crepant boundary on $Y$ is $B_Y$ for both $(X,B+\bM)$ and $(X,B+\bM_t)$. Hence, $a_E(X,B,\bM_t)=a_E(X,B,\bM)>0$ for every prime divisor $E$ over $X$. If $T_Y$ is semiample, then Proposition \ref{prop:absorption}, with $P_Y=T_Y$ and $N_Y=0$, gives a klt pair $(X,\Gamma)$ with $\Gamma\geq B$ and $K_X+\Gamma\sim_{\K}(1+t)(K_X+B+M_X)$. If $T_Y$ is nef and big, then Lemma \ref{lem:big-replacement}(1) for $(X,B+\bM_t)$ gives a klt pair $(X,\Gamma)$ satisfying $\Gamma\geq B$ and $K_X+\Gamma\sim_{\K}(1+t)(K_X+B+M_X)$. By \cite[Theorem 2.7]{DW22}, the projective log resolutions exist. 

For (2), Lemma \ref{lem:big-replacement}(2) gives a klt pair $(X,\Gamma)$ with $\Gamma\geq B$ and $K_X+\Gamma\sim_{\K}c(K_X+B+M_X)$ for some $c\in\Q_{>1}$. In parts (1) and (2), there are $c\in\K_{>0}$ and a $\K$-principal divisor $\Psi_X$ such that
\begin{equation}\label{eq:boundary-reduction}
K_X+\Gamma=c(K_X+B+M_X)+\Psi_X.
\end{equation}
In particular, $K_X+\Gamma$ is pseudoeffective.

Let $k_X\coloneqq H^0(X,\OO_X)$. Since $X$ is integral and projective over $k$, the field extension $k_X/k$ is finite. In particular, $k_X$ is $F$-finite. The morphism $X\to\Spec k_X$ is a projective contraction. By \cite[Theorem 1.6]{DW22}, there is a $\Q$-factorial minimal model $(X',\Gamma')$ with a birational map $\phi\colon X\dashrightarrow X'$. 

Choose a regular projective common model $p\colon W\to X$ and $q\colon W\to X'$ dominating a model of descent of $\bM$. Let $R_W\coloneqq p^*(K_X+\Gamma)-q^*(K_{X'}+\Gamma')$. For a prime divisor $E$ on $X$ not contracted by $\phi$, $\mult_E p_*R_W=0$. For a prime divisor $E$ on $X$ contracted by $\phi$, the definition of a minimal model gives $\mult_E p_*R_W=A_{X',\Gamma'}(E)-A_{X,\Gamma}(E)>0$. Consequently, $p_*R_W\geq0$. For every curve $C$ contracted by $p$, $(-R_W)\cdot C=q^*(K_{X'}+\Gamma')\cdot C\geq0$. The negativity lemma gives $R_W\geq0$.

If a prime divisor $E$ on $X'$ is extracted by $\phi$, then $\mult_E\Gamma'=1$. For the strict transform $E_W$ on $W$, the definition of $R_W$ gives
\[
\mult_{E_W}R_W =A_{X',\Gamma'}(E)-A_{X,\Gamma}(E) =-A_{X,\Gamma}(E)<0,
\]
contrary to $R_W\geq0$. Thus, $\phi$ is a birational contraction and $\Gamma'=\phi_*\Gamma$. Since $W$ is regular and $R_W\geq0$, the pullback of $R_W$ to every higher model is effective. Therefore, $A_{X',\Gamma'}(E)\geq A_{X,\Gamma}(E)>0$ for every prime divisor $E$ over $X$. Hence, $(X',\Gamma')$ is klt. Since $\phi$ is a birational contraction and $\Gamma'=\phi_*\Gamma$, the coefficient of $R_W$ along the strict transform of every prime divisor on $X'$ is zero. Hence, $q_*R_W=0$. Moreover, $R_W$ has positive coefficient along the strict transform of every prime divisor on $X$ contracted by $\phi$.

Let $B'\coloneqq\phi_*B$. Write $\Psi_X=\sum_jc_j\divisor_X(u_j)$ and put $\Psi_{X'}\coloneqq\sum_jc_j\divisor_{X'}(u_j)$. Pushing forward \eqref{eq:boundary-reduction} gives $K_{X'}+\Gamma'=c(K_{X'}+B'+M_{X'})+\Psi_{X'}$. Consequently, $K_{X'}+B'+M_{X'}$ is $\K$-Cartier and nef. Since $p^*\Psi_X=q^*\Psi_{X'}$, we have
\begin{equation}\label{eq:minimal-transfer}
R_W=c\bigl(p^*(K_X+B+M_X)-q^*(K_{X'}+B'+M_{X'})\bigr).
\end{equation}
Let $M_W$ be the trace of $\bM$ on $W$, and define $B_W$ and $B'_W$ by
\begin{align*}
K_W+B_W+M_W=p^*(K_X+B+M_X),\quad K_W+B'_W+M_W=q^*(K_{X'}+B'+M_{X'}).
\end{align*}
Combining the two crepant equalities with \eqref{eq:minimal-transfer} gives $B_W-B'_W=\frac{1}{c}R_W$. Thus,
\begin{equation}\label{eq:generalized-increase}
a_E(X',B',\bM)-a_E(X,B,\bM)
=\frac{1}{c}\ord_E R_W\geq0.
\end{equation}
For every prime divisor on $X$ contracted by $\phi$, the inequality in \eqref{eq:generalized-increase} is strict. Thus, $(X',B'+\bM)$ is a generalized klt minimal model.

For the good minimal model and finite generation conclusions in (2), assume that $\K=\Q$. Then $K_X+B+M_X$ and $K_{X'}+B'+M_{X'}$ are $\Q$-Cartier. The divisor $K_{X'}+\Gamma'$ is nef and big. Let $L'\coloneqq2(K_{X'}+\Gamma')$. Then $L'$ and $L'-(K_{X'}+\Gamma')=K_{X'}+\Gamma'$ are nef and big. By applying \cite[Theorem 1.4]{DW22} to the projective contraction $X'\to\Spec k_X$, the divisor $L'$ is semiample, and hence $K_{X'}+\Gamma'$ is semiample. Since $K_{X'}+\Gamma'\sim_{\Q}c(K_{X'}+B'+M_{X'})$, $K_{X'}+B'+M_{X'}$ is semiample.

If $K_X+B+M_X$ is nef and big, then $K_X+\Gamma\sim_{\Q}c(K_X+B+M_X)$ is nef and big. Let $L\coloneqq2(K_X+\Gamma)$. Then $L$ and $L-(K_X+\Gamma)=K_X+\Gamma$ are nef and big. By applying \cite[Theorem 1.4]{DW22} to the projective contraction $X\to\Spec k_X$, the divisor $L$ is semiample; hence $K_X+\Gamma$, and therefore $K_X+B+M_X$, is semiample.

For finite generation, let $F_W\coloneqq\frac{1}{c}R_W$. Then $p^*(K_X+B+M_X)=q^*(K_{X'}+B'+M_{X'})+F_W$, where $F_W\geq0$ is $q$-exceptional. Choose an integer $r>0$ such that $r(K_X+B+M_X)$, $r(K_{X'}+B'+M_{X'})$, and $rF_W$ are Cartier. For every $\ell\geq0$, normality of $X'$ and exceptionality of $F_W$ give $q_*\OO_W(\ell rF_W)=\OO_{X'}$. Tensoring with $\OO_{X'}(\ell r(K_{X'}+B'+M_{X'}))$ and taking sections gives
\begin{align*}
H^0\!\left(X,\OO_X\bigl(\ell r(K_X+B+M_X)\bigr)\right)
&\simeq H^0\!\left(W,\OO_W\bigl(\ell rp^*(K_X+B+M_X)\bigr)\right)\\
&\simeq H^0\!\left(X',\OO_{X'}\bigl(\ell r(K_{X'}+B'+M_{X'})\bigr)\right).
\end{align*}
Both isomorphisms are induced by the identity on the common function field $k(X)=k(X')$ and are compatible with multiplication. Hence,
\[
R\!\left(X,r(K_X+B+M_X)\right) \simeq R\!\left(X',r(K_{X'}+B'+M_{X'})\right).
\]
For sufficiently divisible $r$, semiampleness and Stein factorization give a projective contraction $h\colon X'\to Z$ and an ample Cartier divisor $A_Z$ satisfying $r(K_{X'}+B'+M_{X'})\sim h^*A_Z$ and $h_*\OO_{X'}=\OO_Z$. Therefore, $R(X',r(K_{X'}+B'+M_{X'}))\simeq \bigoplus_{\ell\geq0}H^0(Z,\OO_Z(\ell A_Z))$, and the section ring is finitely generated by \cite[Tag 0B5T(1)]{Stacks}.
\end{proof}

Corollaries \ref{cor:arbitrary-field-semiample} and \ref{cor:lc-semiample} give two semiampleness consequences of the boundary constructions.

\begin{corollary}\label{cor:arbitrary-field-semiample}
Let $(X,B+\bM)$ be a projective generalized klt threefold over an arbitrary field $k$ of characteristic $p>5$, with $B$ a $\Q$-divisor and $\bM$ a nef $\Q$-Cartier b-divisor. Suppose that $K_X+B+M_X$ is $\Q$-Cartier, nef and big and that a Kodaira decomposition on a model of descent admits the common projective log resolution required in Lemma \ref{lem:big-replacement}(2). Then $K_X+B+M_X$ is semiample.
\end{corollary}

\begin{proof}
Lemma \ref{lem:big-replacement}(2) gives a klt pair $(X,\Gamma)$ with $K_X+\Gamma\sim_{\Q}c(K_X+B+M_X)$ for some $c\in\Q_{>1}$. Let $L\coloneqq2(K_X+\Gamma)$. The divisors $L$ and $L-(K_X+\Gamma)=K_X+\Gamma$ are nef and big. Set $k_X\coloneqq H^0(X,\OO_X)$. The morphism $X\to\Spec k_X$ is a projective contraction. By applying \cite[Theorem 1.4]{DW22} to this contraction, the divisor $L$ is semiample. Hence, $K_X+\Gamma$ and $K_X+B+M_X$ are semiample.
\end{proof}

\begin{corollary}\label{cor:lc-semiample}
Let $(X,B+\bM)$ be a projective generalized lc threefold over an $F$-finite field $k$ of characteristic $p>5$, with $B$ a $\Q$-divisor and $\bM$ a b-semiample nef $\Q$-Cartier b-divisor. If $K_X+B+M_X$ is nef and big, then $K_X+B+M_X$ is semiample.
\end{corollary}

\begin{proof}
Choose a model of descent $\pi\colon Y\to X$ with $M_Y$ semiample. By \cite[Theorem 2.7]{DW22}, $(Y,B_Y)$ admits a projective log resolution. Proposition \ref{prop:absorption} gives an lc pair $(X,\Gamma)$ with $K_X+\Gamma\sim_{\Q}K_X+B+M_X$. The morphism $X\to\Spec H^0(X,\OO_X)$ is a projective contraction. Hence, \cite[Theorem 5.4]{Pos23} gives $K_X+\Gamma$ semiample, and $K_X+B+M_X$ is semiample.
\end{proof}

\bibliographystyle{habbvr}
\bibliography{biblio}

\end{document}